\documentclass{amsart}
\usepackage{amssymb}
\usepackage{amsmath}
\usepackage{amsfonts}
\usepackage{graphicx}
\usepackage{pst-node}
\theoremstyle{plain}
\newtheorem{thm}{Theorem}[section]

\newtheorem{lmm}[thm]{Lemma}
\newtheorem{cor}[thm]{Corollary}

\newtheorem{rmk}[thm]{Remark}

\newtheorem{prb}[thm]{Problem}
\theoremstyle{remark}

\def\pmc#1{\setbox0=\hbox{#1}
    \kern-.1em\copy0\kern-\wd0
    \kern.1em\copy0\kern-\wd0}

\def\De{\Delta}

\def\vep{\varepsilon}

\def\si{\sigma}
\def\Si{\Sigma}

\def\op{\operatorname}
\def\ov{\overline}

\def\sm{\setminus}
\begin{document}

\title
{On Snake cones, alternating cones and related constructions}

\author[K. Eda]{Katsuya Eda}
\address{School of Science and Engineering,
Waseda University, Tokyo 169-8555, Japan}
\email{eda@logic.info.waseda.ac.jp}

\author[U.H. Karimov]{Umed H. Karimov}
\address{Institute of Mathematics,
Academy of Sciences of Tajikistan, Ul. Ainy $299^A$, Dushanbe
734063, Tajikistan} \email{umedkarimov@gmail.com}

\author[D. Repov\v s]{Du\v san Repov\v s }
\address{Faculty of Education, and
Faculty of Mathematics and Physics, University of Ljublja -na, P.O.Box 2964,
Ljubljana 1001, Slovenia} \email{dusan.repovs@guest.arnes.si}

\author[A. Zastrow]{Andreas Zastrow}
\address{Institute of Mathematics, Gdansk University, ul. Wita Stwosza 57,
80-952 Gda\'nsk, Poland}
\email{zastrow@mat.ug.edu.pl}

\begin{abstract}
We show that the Snake  on a square $SC(S^1)$ is homotopy equivalent to the space
 $AC(S^1)$ which was investigated in the previous work by Eda, Karimov
 and Repov\v s. We also introduce related constructions $CSC( - )$ and 
$CAC( - )$ and investigate homotopical differences between these four 
 constructions. Finally, we explicitly describe the second
 homology group of the Hawaiian tori wedge.  
\end{abstract}
\date{\today}
\maketitle
\date{\today}
\section{Introduction}
The functor $SC(-,-)$, mapping
from the category of all spaces with base points and continuous
mappings to the subcategory of simply connected spaces was constructed 
in \cite{EKR:topsin}. 
We named $SC(X,x)$
the {\it Snake cone} over a pointed space $(X,x)$. In the case
when the space $X$ is a circle $S^1$ with a base point $x$, the resulting space
$SC(S^1, x)$, called the {\it Snake on a square}, is a cell-like simply connected 
2-dimensional Peano
continuum \cite{EKR:topsin}.
It was shown in \cite{EKR:nonaspherical} that the space $SC(S^1,
x)$ is not only noncontractible but is also nonaspherical (because the second
homotopy group of this space is nontrivial, see also
\cite{EKR:secondhomotopy, EKR:snake}). 
We investigated another functor $AC(-,-)$ in
\cite{EKR:simple}, which shares many properties with $SC(-,-)$, and
we proved that
$AC(\mathbb{H},o)$ is not homotopy equivalent to $SC(\mathbb{H},o)$ for
the Hawaiian earring $\mathbb{H}$ with the distinguished point $o$. 
 We named $AC(X,x)$ the {\it Alternating cone} over a pointed space $(X,x)$. 
In the present paper we shall introduce some variants of these 
constructions, i.e. 
the {\it Collapsed snake cone} $CSC(X,x)$ and the {\it Collapsed alternating 
cone} $CAC(X,x)$, and we shall investigate homotopy equivalences among these 
four functors. 

Our main results are the following: 

\begin{thm}\label{thm:main1} 
If a space $X$ is semi-locally strongly contractible at $x_0\in X$, then
$SC(X,x_0)$ and $CSC(X,x_0)$ (resp. $AC(X,x_0)$ and $CAC(X,x_0)$) 
are also homotopy equivalent. 
\end{thm}
\begin{thm}\label{thm:main2}
The Snake on a square $SC(S^1, x_0)$ and the Alternating cone $AC(S^1,x_0)$ are 
homotopy equivalent for every $x_0\in S^1$.
\end{thm}
\begin{thm}\label{thm:main3}
For the Hawaiian earring $\mathbb{H}$ with the base point $o$ 
the following properties hold: 
\begin{itemize}
\item[(1)] $SC(\mathbb{H},o)$ and $AC(\mathbb{H},o)$ (resp. $SC(\mathbb{H},o)$ 
           and $CSC(\mathbb{H},o)$) are not homotopy equivalent; 
\item[(2)] $AC(\mathbb{H},o)$ and $CAC(\mathbb{H},o)$ are not homotopy
	   equivalent; but 
\item[(3)] $CSC(\mathbb{H},o)$ and $CAC(\mathbb{H},o)$ are homotopy
	   equivalent.
\end{itemize}
\end{thm}
\begin{thm}\label{thm:main4} 
For the 2-dimensional torus $T$ with the base point 
$z_0\in T$, the spaces $SC(T,z_0)$ and
 $AC(T,z_0)$ are not homotopy equivalent and consequently, also the spaces  
 $CSC(T,z_0)$ and $CAC(T,z_0)$ are not homotopy equivalent. 
\end{thm}
Consequently, we get the following: 
\begin{cor}
For each pair of functors $SC(-,-)$, $AC(-,-)$, $CSC(-,-)$ and
 $CAC(-,-)$, there exists a space such that the resulting functorial 
spaces are not
 homotopy equivalent. 
\end{cor}
We shall define the {\it Hawaiian tori wedge} similarly to the Hawaiian 
earring by
replacing the circle by the torus. A precise definition 
and supporting notions will be given in the forthcoming sections. 
\begin{thm}\label{thm:main5} 
Let $\mathbb{H}_T$ be the Hawaiian tori wedge. Then the following 
properties hold: 
\begin{enumerate}
\item $\pi _1(\mathbb{H}_T)$ is isomorphic to the free $\si$-product of
      countable copies of the free abelian group of rank two; 
\item $\pi _2(\mathbb{H}_T)$ is trivial; and  
\item $H_2(\mathbb{H}_T)$ is isomorphic to the free abelian group on 
      countably many generators. The generators are associated with the 
      fundamental cycles of the tori. 
\end{enumerate}
\end{thm}
This contrasts with the known results concerning the 2-dimensional Hawaiian
earring $\mathbb{H}_2$. Namely, 
\begin{enumerate}
\item $\pi _1(\mathbb{H}_2)$ is trivial \cite[Theorem A.1]{E:free}; and 
\item $\pi _2(\mathbb{H}_2) \cong H_2(\mathbb{H}_2)$ is isomorphic to
      the direct product of countably many copies of $\mathbb{Z}$ 
\cite[Corollary 1.2]{EK:aloha}. 
\end{enumerate}
Throughout this paper $X$ stands for a path-connected compact 
Hausdorff space. 
Standard notions are undefined and we refer the reader to
\cite{Spanier:algtop}.

\section{The construction of the Snake cone $SC(X,x)$, 
the Alternating cone $AC(X,x)$ and their variations}
In this paper we shall apply our constructions only to compact spaces, and so
the definitions of topologies, 
which we shall use, 
may look to be different from the original
ones in \cite{EKR:topsin}, but they are in fact the same. 
The construction of the Snake cone is based on the piecewise-linear
Topologist sine
curve $\mathcal{T}$ which is homeomorphic to the usual Topologist sine curve.
To describe this space and for the next discussion we need to fix some
terminology. For any two points $A$ and $B$ in the plane
$\mathbb{R}^2$, we denote by $[A, B]$ the linear segment connecting
these points. The unit segment of the real line is denoted by
$\mathbb{I}.$ The point of the coordinate plane $\mathbb{R}^2$ with
coordinates $a$ and $b$ is denoted as $(a; b)$, 
particularly when we describe 
the points in $SC(X,x)$ and $AC(X,x)$.  
 Let $A = (0; 0), B
= (0; 1), A_n = (1/n; 0), B_n = (1/n; 1)$  be points and let $L =
[A, B],$ $L_{2n-1} = [A_n, B_n],$ $L_{2n} = [B_n, A_{n+1}]$ be the
segments in the plane $\mathbb{R}^2$ for $n \in \mathbb{N} = \{1,
2, 3, \dots \}.$
We also let $C_{2n-1} = (1/n; 1/2)$ and $C_{2n} = (1/(n+1) + 1/2n(n+1); 1/2)$
for $n\in \mathbb{N}$ and $C =(0;1/2)$, see Figure 1.  

The piecewise linear Topologist sine curve $\mathcal{T}$ is the subspace of 
$\mathbb{R}^2$ defined as the union of $L_n$ and $L.$

\smallskip
\begin{figure}
\begin{center}
\includegraphics[scale=0.85]{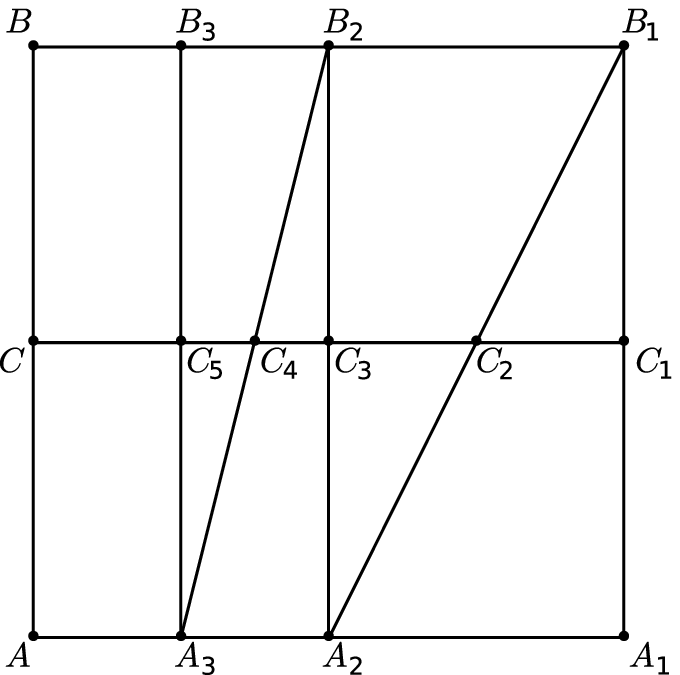} 
\smallskip

Figure 1: Topologist sine curve
\end{center}
\end{figure}
The Snake cone $SC(X,x)$ over a compact pointed space $(X,x)$ is 
the quotient space of the topological sum $(X\times \mathcal{T})
\bigsqcup \mathbb{I}^2$ via the identification of the points
$(x,t)\in X\times \mathcal{T}$ with $t\in \mathcal{T} \sm L \subset
\mathbb{I}^2$ and the identification of each set $X\times\{t\}$ 
with the point $t$, for every $t\in L$ (\cite{EKR:topsin}).

Define the following closed subspace of $X\times \mathbb{I}^2$  
\[
 Y = X\times \{ 0\}\times \mathbb{I}
\cup \bigcup _{n\in \mathbb{N}}X\times \{ 1/n\}\times \mathbb{I}
\cup \{ x\}\times \mathbb{I}\times \mathbb{I}.
\]

The Alternating cone $AC(X, x)$ over a compact pointed space $(X,x)$ is 
defined as 
the quotient space of $Y$ via the identification of each set $X \times 
\{ 0\}\times \{ y\}$ to $(0; y)$, $X\times \{1/(2n-1)\}\times \{ 0\}$ to 
$A_{2n-1} = (1/(2n-1); 0)$ 
and $X\times \{1/2n\}\times \{ 1\}$ to $B_{2n} =(1/2n; 1)$ for each
$y\in \mathbb{I}$ and each $n\in \mathbb{N}$, respectively.

In both cases of $SC(X,x)$ and $AC(X,x)$ 
let $p:SC(X, x) \to \mathbb{I}^2$ or $p:AC(X, x) \to \mathbb{I}^2$ be
the natural projection and define $p_1$ and $p_2$ by $p(u) = (p_1(u); 
p_2(u))$. 

The spaces $CSC(X,x)$ and $CAC(X,x)$ are obtained from spaces 
$SC(X,x)$ and $AC(X,x)$, respectively by identifying each 
point $(a; b)\in \mathbb{I}^2$ with $(0; b)$ for all $a,b\in \mathbb{I}$,
i.e. by collapsing $\mathbb{I}^2$ to $L$. We also denote this distinguished 
interval by $L$. 
The projection $p_2$ is defined on all the spaces $SC(X,x), AC(X,x), 
CSC(X,x)$, and $CAC(X,x)$, while $p_1$ is defined only on 
$SC(X,x)$ and $AC(X,x)$.

\smallskip
\begin{figure}
\begin{center}
\includegraphics[scale=0.85]{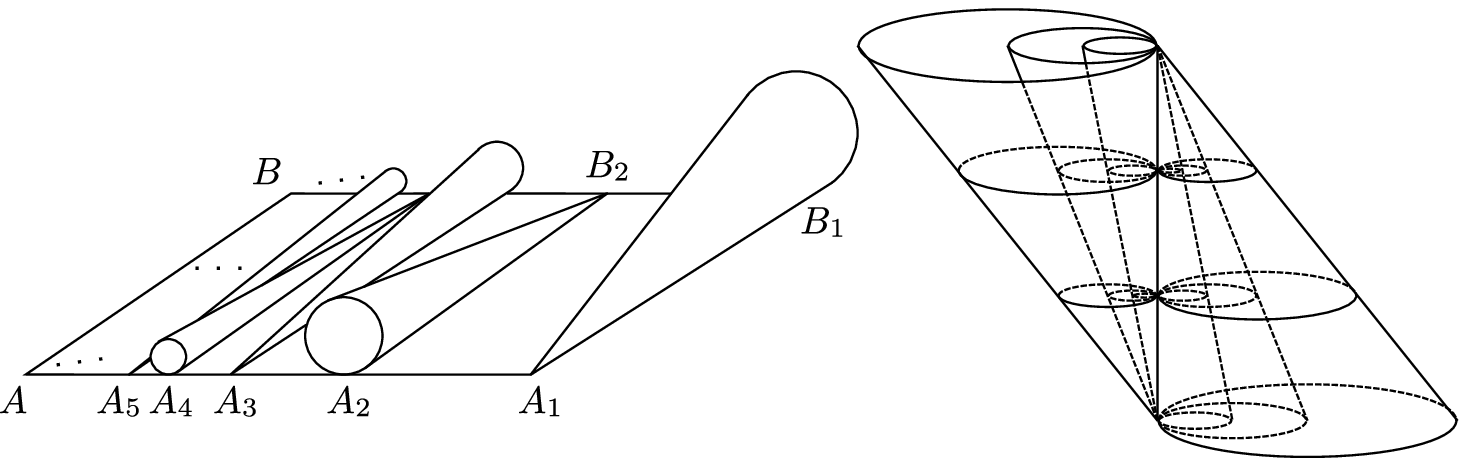} 
\smallskip

Figure 2: $AC(S^1) \mbox{ and } CAC(S^1)$ 
\end{center}
\end{figure}

For $m,n\in \mathbb{N}$ with $m\le n$, define $SC_{m,n}(X,x) =
p_1^{-1}([1/n,1/m])$ and $SC_{m}(X,x) = p_1^{-1}([0,1/m])$ 
when $p$ is mapping from $SC(X,x)$ to $\mathbb{I}^2$,  
and define 
$AC_{m,n}(X,x) = p_1^{-1}([1/n,1/m])$ and $AC_{m}(X,x)
= p_1^{-1}([0,1/m])$ when $p$ is mapping from $AC(X,x)$ to $\mathbb{I}^2$. 

For a subspace $S$ of these spaces and a map 
$f$ defined on $S$ whose range is one of these spaces, $f$ is called 
{\it flat}, if $p_2(u) = p_2(v)$ implies $p_2(f(u)) = p_2(f(v))$ for
$u,v\in S$. 
Similarly, a homotopy 
$H:S\times \mathbb{I}\to Z$ where $Z$ is one of these four spaces, is 
said to be flat, if the map $H(-,t)$ is flat for each $t\in \mathbb{I}$.  

For a subspace $S$ of these spaces and a map 
$f$ defined on $S$ whose range is one of these spaces, $f$ is called 
{\it level-preserving} (resp., {\it $\vep$-level-preserving}), if 
$p_2(f(u)) = p_2(u)$ (resp., $|p_2(f(u)) - p_2(u)| <\vep$) for every $u\in
S$. Similarly, a homotopy 
$H:S\times \mathbb{I}\to Z$ where $Z$ is one of these four spaces, is 
{\it level-preserving} (resp., {\it $\vep$-level-preserving}), if 
$p_2(H(u),t) = p_2(u)$ (resp., $|p_2(H(u,t)) - p_2(u)| <\vep$) for every $u\in
S$ and $t\in\mathbb{I}$. 

The {\it cone} over a space $X$, denoted by $C(X)$, is defined as the 
quotient space of
$X\times \mathbb{I}$ by identifying $X\times \{ 1\}$ to a point and so 
we can describe points on the cone by points on $X\times \mathbb{I}$. 

When arguments are not related to base points, we shall abbreviate 
$SC(X,x)$ by $SC(X)$, and so on. 

Finally, we introduce a construction of spaces for our further
investigation. 
For spaces $X_n$ $(n\in\mathbb{N})$ with $x_n\in X_n$, let 
$\widetilde{\bigvee}_{n\in\mathbb{N}}(X_n,x_n)$ be the space obtained by 
identifying all $x_n$'s to the point $x^*$ so that every neighborhood of $x^*$ 
contains almost all $X_n$'s and each subspace topology of $X_n$ coincides 
with the topology of $X_n$.     
When the index set is finite, we write $\bigvee_{n=1}^k(X_n,x_n)$ as
usual. 
When each $X_n$ is a copy of the circle,  
$\widetilde{\bigvee}_{n\in\mathbb{N}}(X_n,x_n)$ is homeomorphic to the
Hawaiian earring. When each $X_n$ is a copy of the 2-dimensional torus $T$,
we call $\widetilde{\bigvee}_{n\in\mathbb{N}}(X_n,x_n)$ the {\it Hawaiian
tori wedge} and denote it by $\mathbb{H}_T$. 
\begin{rmk}
{\rm The {\it Sombrero space} was introduced and studied in \cite{CMRZZ}. 
The {\it piecewise linear Sombrero space} is defined as the subspace of
$\mathbb{R}^3$ obtained by rotating $\mathcal{T}$ about its limiting
interval $L \subseteq \{ 0\}\times \{ 0\}\times \mathbb{R}$. The {\it Sliced 
Sombrero space} is then defined as the union of the Sombrero space and 
$\mathbb{I}\times \{ 0\} \times
\mathbb{I}$. In other words, the Sombrero space is the quotient space of the
product $S^1\times \mathcal{T}$ that results from identifying the circles 
$S^1\times \{ a\}$ for $a\in L\subset \mathcal{T}$ to one point, and the  
Sliced Sombrero space is the quotient space of the topological sum of
the Sombrero space and the unit square $\mathbb{I}^2$ that results from 
identifying the two topological sine-curves $\mathcal{T}$ which have been 
defined in each of the components of this topological sum. 

In this form one can see that the Sliced Sombrero and the Snake on a square
 $SC(S^1)$ are analogously built and thus are homeomorphic. 
Our starting point of the investigations in this
 paper was the discovery of this fact and the homotopy equivalence
 between the Sliced Sombrero and $AC(S^1)$, which is reflected in 
 Theorem~\ref{thm:main2}.  

While describing the proof of Theorem~\ref{thm:main2}, we found the 
constructions $CSC(X,x)$ and $CAC(X,x)$. Though differences among 
the constructions $SC(X,x)$, $AC(X,x)$, $CSC(X,x)$ and $CAC(X,x)$, except 
the difference between $CSC(X,x)$ and $CAC(X,x)$, were 
shown using the Hawaiian earring $\mathbb{H}$, the difference 
was left open. Similarly as for the circle 
$S^1$, $SC(S^n)$ and $AC(S^n)$ are homotopy equivalent for 
the $n$-sphere $S^n$, which was meanwhile shown by the 
first author. We found that $SC(T)$ and $AC(T)$ are not homotopy 
equivalent for the torus $T$ and by Theorem~\ref{thm:main1} 
$CSC(T)$ and $CAC(T)$ are not homotopy equivalent, either. 
Since a space which is homotopy equivalent to the Hawaiian tori wedge, appears 
in these spaces and it works in the proof, our interests turned 
to $H_2(\mathbb{H}_T)$. This explains how our results in this paper 
are related. }   

\end{rmk}
\section{Proof of Theorems~\ref{thm:main1} and \ref{thm:main2}}
First we give another 
presentation of the spaces $CSC(X,x)$ and $CAC(X,x)$. 
Let $X_n$ be a copy of $X$ and $x_n$ a copy of $x_0\in X$ for each $n$. 
Then $CSC(X,x)$ is homeomorphic to the quotient space of 
$\widetilde{\bigvee}_{n\in\mathbb{N}}(X_n,x_n)\times \mathbb{I}$,  
obtained by identifying 
$X_{2n}\times \{ 0\}$ with $X_{2n-1}\times \{ 0\}$, and by identifying 
$X_{2n+1}\times \{ 1\}$ with $X_{2n}\times \{ 1\}$ for $n\in\mathbb{N}$. 

To present $CAC(X,x)$ in another way, let 
$C(\widetilde{\bigvee}_{n\in\mathbb{N}}(X_{2n-1},x_{2n-1}))$ be the 
cone over $\widetilde{\bigvee}_{n\in\mathbb{N}}(X_{2n-1},x_{2n-1})$, 
where $x^*$ is the point of identification and 
$C(\widetilde{\bigvee}_{n\in\mathbb{N}}(X_{2n},x_{2n}))$ with $x^{**}$ 
analogously. We denote the interval connecting the vertex of the first cone and 
the base point $x^*$ by $\{ x^*\}\times \mathbb{I}$, where $(x^*,1)$ denotes the vertex and $(x^*,0)$ denotes the point $x^*$ in the base space 
$\widetilde{\bigvee}_{n\in\mathbb{N}}(X_{2n-1},x_{2n-1})$. 
We denote the corresponding interval by $\{ x^{**}\}\times \mathbb{I}$, 
analogously for the second cone. 
Then $CAC(X,x_0)$ is homeomorphic to the quotient space of the disjoint union 
$C(\widetilde{\bigvee}_{n\in\mathbb{N}}(X_{2n-1},x_{2n-1}))\sqcup 
C(\widetilde{\bigvee}_{n\in\mathbb{N}}(X_{2n},x_{2n}))$ via the identification 
of $(x^*,t)$ with $( x^{**},1-t)$ for $t\in \mathbb{I}$.  

A space $X$ is called {\it semi-locally strongly
contractible} at $x_0\in X$, if there exists a neighborhood $U$ of $x$ and
a continuous map $H: U\times \mathbb{I}\to X$ such that $H(u,1) = u$,
$H(u,0) = x_0$ and $H(x_0,t) = x_0$, for every $u\in U$ and $t\in
\mathbb{I}$. 
\medskip

{\it Proof of\/} Theorem~\ref{thm:main1}.
By the semi-local strong contractibility 
at $x_0\in X$, 
we have $U$ and $H$ as above. 
Choose neighborhoods $V_0$ and $V_1$ of $x_0$ and continuous maps $F,G:X\to
\mathbb{I}$ such that $\ov{V_0}\subseteq V_1,  \ov{V_1}\subseteq U$,
$F(x_0) = 0$ and $F(x) = 1$ for $x\notin V_0$ and 
$G(x) = 0$ for $x\in \ov{V_0}$ and $G(x) = 1$ for $x\notin V_1$. 
We denote the quotient map from $SC(X,x_0)$ to $CSC(X,x_0)$ by $\si$. 
Observe that the restriction of $\si$ to $p^{-1}(\mathcal{T})\sm \mathcal{T}$ 
is a bijection onto $CSC(X,x_0)$. 
We define $\tau : CSC(X,x_0)\to SC(X,x_0)$ by: 
\[
 \tau (\si (P,x)) = \left\{ 
\begin{array}{ll}
(F(x)\cdot p_1(P); p_2(P)), \quad&\mbox{ for }x\in V_0 \\
(P,H(x,G(x))), \quad&\mbox{ for }x\in V_1\sm V_0 \\ 
(P,x), \quad&\mbox{ for }x\notin V_1 
\end{array}
\right.
\]
for $P\in \mathcal{T}$ and $x\in X$. 
Since every element of $ CSC(X,x_0)$ can be expressed as $\si (P,x)$ and $(F(x_0)\cdot p_1(P);p_2(P)) = (0; p_2(P))$, $\tau$ is well-defined. 
Points from copies of $V_0$ are mapped by $\tau$ into the square 
of $SC(X,x_0)$ 
and points from copies of $X\sm V_1$ are mapped by $\tau$ into copies of $X$ 
in $SC(X,x_0)$. Therefore, using only level-preserving homotopies 
we can show that $\tau\circ \si$ is homotopic to $\op{id}_{SC(X,x_0)}$ 
and that $\si \circ\tau$ is also homotopic to $\op{id}_{CSC(X,x_0)}$. 

We use the same maps $F$, $G$ and $H$, and also denote by $\si$
the quotient map from
 $AC(X,x_0)$ to $CAC(X,x_0)$. When we use the presentation of 
$CAC(X,x_0)$ via the cones 
$C(\widetilde{\bigvee}_{n\in\mathbb{N}}(X_{2n-1},x_{2n-1}))$ and 
$C(\widetilde{\bigvee}_{n\in\mathbb{N}}(X_{2n},x_{2n}))$, 
$V_{0,n}$, $V_{1,n}$ and $U_n$ are copies of $V_0$, $V_1$ and $U$ on $X_n$,  
respectively and we also use $F$, $G$ and $H$ for the copies. 
\smallskip
\begin{center}
 \begin{tabular}{ c c c}
 \includegraphics[scale=0.85]{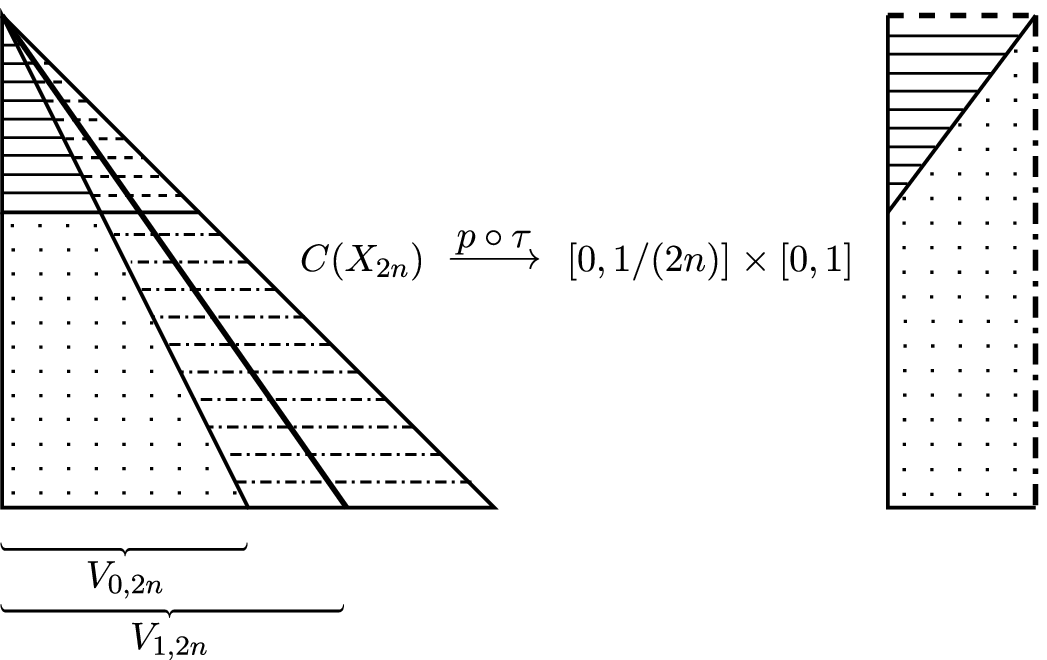} & \hspace{15 mm} &
 \end{tabular}
 \\
 \medskip
Figure 3: The figure visualizes the mapping $\tau$. The hatching styles 
indicate to which regions the corresponding domains are mapped. Observe that 
only $p\circ \tau$ is pictured; in particular the $\tau$-image of the 
dash-dotted area is considerably bigger than it is visible from the figure.   
\end{center}

The homotopy equivalence between $AC(X,x_0)$ and $CAC(X,x_0)$ is proved
similarly, but more care is necessary, because the homotopy inverse $\tau$ 
of $\si$, which we shall define in the sequel, is not level-preserving. We 
define $\tau$ only on
$C(\widetilde{\bigvee}_{n\in\mathbb{N}}(X_{2n},x_{2n}))$ for
notational convenience, but the other case is similar. Let $\tau (x,s) = $ 
\[
 \left\{ 
\begin{array}{ll}((1\! -\! s)F(x); s+(1\! -\! s)F(x)) &\mbox{ for }(2n\! -\! 1)/(2n)\le
 s\le 1, x\in X_{2n} \\
(F(x)/(2n) ; s(1+ F(x)/(2n\! -\! 1)))
&\mbox{ for }0\le s < (2n\! -\! 1)/(2n), x\in V_{0,2n} \\
(H(x, G(x)), 2ns/(2n\! -\! 1)) &\mbox{ for }0\le s < (2n\! -\! 1)/(2n), x\in
 V_{1,2n}\sm V_{0,2n} \\
(x, 2ns/(2n\! -\! 1)) &\mbox{ otherwise.}
\end{array}
\right. 
\]
Figure 3 shows the restriction of $p\circ\tau$ to $C(X_{2n})$. 
We remark that $\si (0; t) = (x^*,t)$ and $\tau (x^*,t) = (0;t)$ and 
$(1-2n)/(2n)$ converge to $1$. Using this information we can see that 
$\tau\circ \si$ is homotopic to 
$\op{id}_{AC(X,x_0)}$ and $\si \circ\tau$ is homotopic to 
$\op{id}_{CAC(X,x_0)}$.  
\qed
\medskip

{\it Proof of Theorem~\ref{thm:main2}}.
By virtue of Theorem~\ref{thm:main1} the fact that $CSC(S^1)$ and $CAC(S^1)$ are homotopy equivalent implies the present theorem.  
 To find this homotopy equivalence, we use embeddings of these spaces in 
$\mathbb{R}^3$. 
We define subsets of these embeddings as follows: 
\begin{eqnarray*}
S_n &\!\! =\!\! & \{ (x,y,z) \; | \; (x-1/n)^2 + y^2 = 1/n^2, z\in\mathbb{I}\}, \\
T_n &\!\! =\!\! & \{ (x,y,z)\; |\; \displaystyle{(x-(\frac{1}{n+1} + \frac{z}{n(n+1)}))^2} + y^2 = \displaystyle{(\frac{1}{n+1} + \frac{z}{n(n+1)})^2}, 
z\in\mathbb{I}\},\\
U_n &\!\! =\!\! & \{ (x,y,z) \; |\; (x-(1-z)/n)^2 + y^2 = (1-z)^2/n^2,
 z\in\mathbb{I} \}, \\ 
V_n &\!\! =\!\! &\{ (x,y,z) \; |\; (x+ z/n)^2 + y^2 = z^2/n^2, z\in\mathbb{I} \},  
\end{eqnarray*}
$S_n^+ = S_n \cap (\mathbb{R}\times \{ y\; |\; y\ge 0\}\times
\mathbb{R})$ and $T_n^+ =  T_n \cap (\mathbb{R}\times \{ y\; |\; y\ge
0\}\times\mathbb{R})$. 

By the preceding description, $CSC(S^1)$ is homotopy equivalent to the
subspace $Y_0 = \bigcup _{n\in\mathbb{N}}(S_n\cup T_n)$ of
$\mathbb{R}^3$. 
Let $Y_1$ be the following subspace of $\mathbb{R}^3$(see Figure 4): 
\[
 [0,2]\times \{ 0\} \times \mathbb{I}\; \cup \bigcup _{n\in
 \mathbb{N}} S_n^+ \cup T_n^+. 
\]
\begin{center}
 \begin{tabular}{ c c c}
 \includegraphics[scale=0.85]{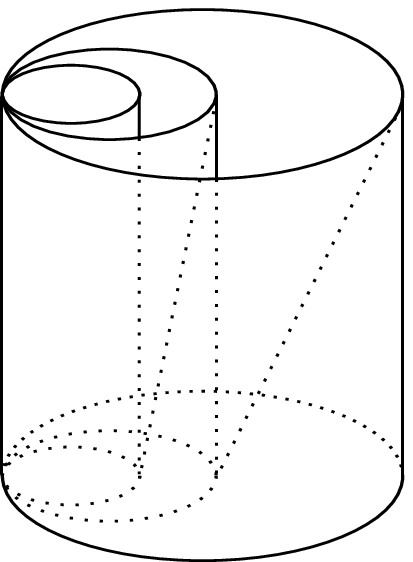} & \hspace{15 mm} &
 \includegraphics[scale=0.85]{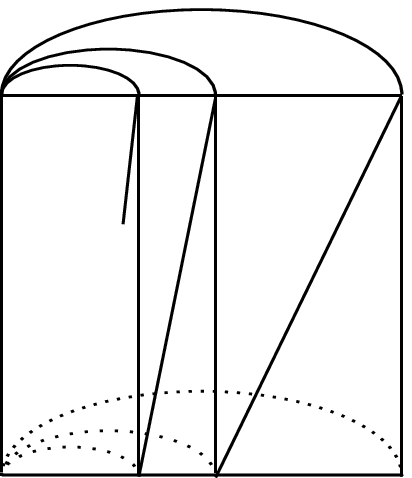} \\
 \end{tabular}
 \\
 \medskip
 Figure 4: $CSC(S^1)$ and $Y_1$
\end{center}

The homotopy equivalence between the spaces 
$\bigcup _{n\in \mathbb{N}}\{ (x,y) \; | \; (x-1/n)^2 + y^2 = 1/n^2\}$ 
and $[0,2]\times \{ 0\}\cup \bigcup _{n\in \mathbb{N}}\{ (x,y) \; |
\; (x-1/n)^2 + y^2 = 1/n^2,  y\ge 0\}$ induces a homotopy 
equivalence between $Y_0$ and
$Y_1$ and it now suffices to establish the homotopy equivalence between $Y_1$
and $CAC(S^1)$. 
We use the presentation of $CAC(S^1)$ 
from the introduction of Section 3
and we realize it in $\mathbb{R}^3$ as the following subspace: 
\[
\bigcup _{n\in \mathbb{N}} U_n \cup V_n. 
\]

Define $\si : Y_1 \to CAC(S^1)$ so that 
$\si (x,y,z) = (0,0,z)$, if
\[
\displaystyle{(x-\frac{1}{n})^2} + y^2 = \displaystyle{(\frac{1}{n})^2}
\]
or if
\[
\displaystyle{(x-(\frac{1}{n+1} + \frac{z}{n(n+1)}))^2} + y^2 = \displaystyle{(\frac{1}{n+1} + \frac{z}{n(n+1)})^2}
\]
for $n\in \mathbb{N}$ and $z\in \mathbb{I}$, and 
so that $\si$ maps level-preserving from  
\[
 \{ (x;y) \; | \; 0\le y\le \displaystyle{\frac{x-1/(n+1)}{n(n+1)}, \frac{1}{n+1}\le x\le \frac{1}{n}}\}
\]
onto $U_n$ and from 
\[
 \{ (x;y) \; | \; \displaystyle{\frac{x-1/(n+1)}{n(n+1)}\le y\le 1, \frac{1}{n+1}\le x\le \frac{1}{n}}\}
\]
onto $V_n$. 

We define $\tau :Y_1\to CAC(S^1)$ so that $\tau (0,0,z) = (0,0,z)$ for 
$z\in \mathbb{I}$ and so that $\tau$ maps $U_n$ homeomorphically 
onto 
\[
 S_n^+ \cup T_n^+ \cup \{ (x,y,z) \; | \; \displaystyle{\frac{1}{n+1}\le x\le 1/n, y=0, 0\le z\le \frac{x- 1/(n+1)}{n(n+1)}}\}
\]
and 
$V_n$ onto 
\[
 S_{n+1}^+ \cup\; T_n^+ \cup 
\{ (x,y,z) \; | \; \displaystyle{\frac{1}{n+1}\le x\le 1/n, y=0, \frac{x- 1/(n+1)}{n(n+1)}\le z\le 1}\}.
\]
We define $\tau$ uniformly with respect to $n$, in particular so that $\tau$ is $(1/n)$-level-preserving on $U_n$ and $V_n$.

\smallskip
\begin{center}
 \begin{tabular}{ c }
 \includegraphics[scale=0.85]{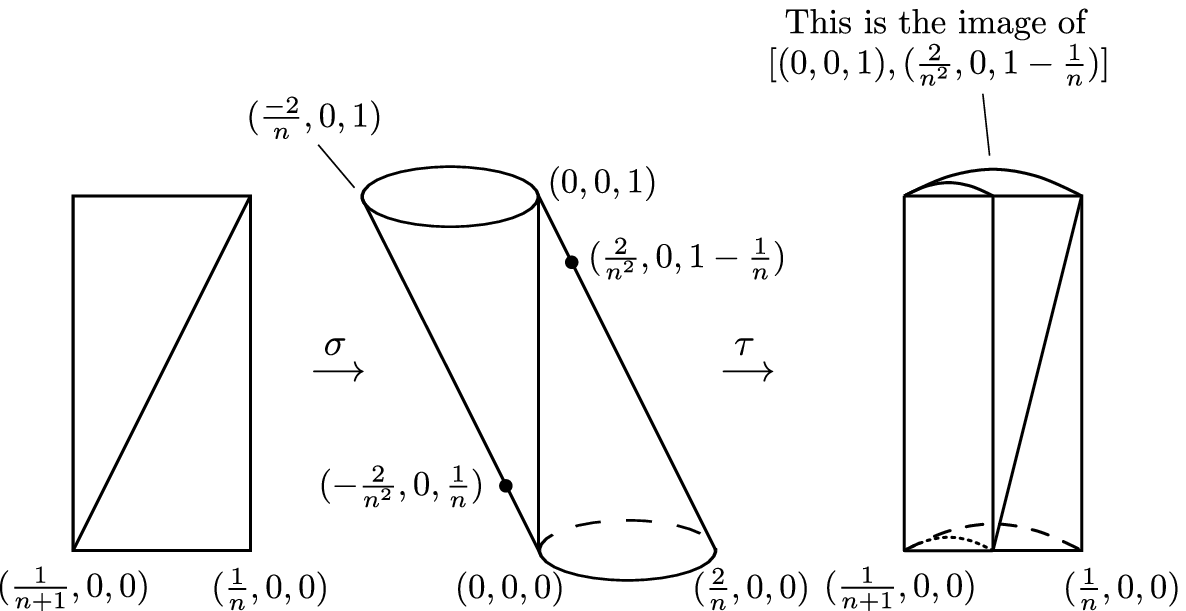} 
 \end{tabular}
 \\
 \smallskip
 Figure 5: Partial illustrations of the maps 
$\si$ and $\tau$. 
\end{center}

In particular,  
$\tau$ maps the segment $[(0,0,1), (2/n^2,1-1/n)]$ onto
the arc $\{ (x,y,1)\; |$ 
$(x-1/n)^2 + y^2 = 1/n^2,y\ge 0\}$ and 
the segment
$[(0,0,0), (-2/n^2, 0,1/n)]$ onto the arc 
$\{ (x,y,0)\; | \; (x-1/(n+1))^2 + y^2 = 1/(n+1)^2,y\ge 0\}$ 
which is the only dotted arc that we have drawn in Figure 5.

Since $[(0,0,1), (2/n^2,0,1-1/n)]$ converge to $(0,0,1)$ and 
$[(0,0,0), (-2/n^2, 0,1/n)]$ converge to $(0,0,0)$ when $n$ tends to 
infinity and $\si$ is level-preserving, and $\tau$ is $(1/n)$-level-preserving 
on $U_n$ and $V_n$, 
we can conclude that $\si\circ \tau$ and $\tau\circ \si$ are homotopic to 
the identity mapping on $CAC(S^1)$ and $Y_1$, respectively. 

\section{Proof of Theorem~\ref{thm:main3}}
The first lemma concerns a certain useful property of $AC(X,x)$ for any 
path-connected space $X$. 
\begin{lmm}\label{lmm:contractible}
Let $X$ be a path-connected space. Then 
$p_2^{-1}([0,1))$ and $p_2^{-1}((0,1])$ are contractible in
 $AC(X,x)$. The same also holds for $CAC(X,x)$. 
\end{lmm}
\begin{proof}
Since $p_2^{-1}(\{ 0\})$ is a strong deformation retract of
 $p_2^{-1}([0,1))$, it suffices to show that $p_2^{-1}(\{ 0\})$ is
 contractible in $AC(X,x)$. Since $p_2^{-1}(\{ 0\})$ deforms to the
 segment $[ (0;1),(1;1)]$, $p_2^{-1}(\{ 0\})$ is contractible in
 $AC(X,x)$. The contractibility of $p_2^{-1}((0,1])$ in
 $AC(X,x)$ can be proved similarly.
In the case of $CAC(X,x)$ the above procedure yields a point instead of 
the segment $[ (0;0),(1;0)]$. 
\end{proof}
We prove (1),(2) and (3) separately.

\subsection{Proof of Theorem 1.3: (1) and (2)}

{\it Proof of\/ Theorem~\ref{thm:main3} (1)}.
The fact that $SC(\mathbb{H},o)$ and $AC(\mathbb{H},o)$ are not homotopy
equivalent was proved in \cite[Theorem 3.6]{EKR:simple},   
where it was established that any embedding of $\mathbb{H}$ to
$AC(\mathbb{H},o)$ is null homotopic, while there exists an essential
embedding of $\mathbb{H}$ to $SC(\mathbb{H},o)$. 

Hence, for the second nonequivalence it suffices to show that any 
embedding of $\mathbb{H}$ to
$CSC(\mathbb{H},o)$ is null homotopic.
Let $f:\mathbb{H}\to CSC(\mathbb{H},o)$ be an embedding. 
Let $x^*$ be the identified points with copies of $o\in \mathbb{H}$. 
Since 
$CSC(\mathbb{H},o)$ is locally contractible at points in
$CSC(\mathbb{H},o)\sm \{ x^*\}\times \mathbb{I}$, $f(o)$ belongs to 
$ \{ x^*\}\times \mathbb{I}$, which is the crucial difference between
$CSC(\mathbb{H},o)$ and $SC(\mathbb{H},o)$. Therefore, except for 
finitely many
circles of $\mathbb{H}$, $f$ maps into $p_2^{-1}([0,1))$ or
$p_2^{-1}((0,1])$. By a similar argument as in the proof of \cite[Lemma
3.3]{EKR:topsin} we can show that $f$ is homotopic to a map $f'$ into 
$p_2^{-1}([0,1))$ or $p_2^{-1}((0,1])$. By a similar
argument as in the proof of \cite[Theorem 1.1]{EKR:topsin} we can then  
produce a null-homotopy of $f'$. 
\qed

For a proof of Theorem~\ref{thm:main3}(2) we need some notions and
lemmas. 
For $AC(Z)$ and $P\in \bigcup _{n\in \mathbb{N}}\{ 1/(2n-1)\} \times
(0,1] \cup \{ 1/(2n) \times [0,1)\}$ , let $Z_P = p^{-1}(\{ P\})$, which
is homeomorphic to $Z$ and $M_s = p_2^{-1}(\{ s \})$ for $s\in \mathbb{I}$. 

An interval $[a,b]$ with $0\le a<b \le 1$ is called an {\it essential interval}
for a flat continuous map $f:AC_n(Z)\to AC(Z)$,
if the following hold: 
\begin{itemize}
\item[(1)] $f$ maps $p_2^{-1}(\{ a\})\cap AC_n(Z)$ into $M_0$ and
	   $p_2^{-1}(\{ b\})$ into $M_1$; 
\item[(2)] $f$ maps $p_2^{-1}((a,b))\cap AC_n(Z)$ into
	   $p_2^{-1}((0,1))$; and
\item[(3)] for every $a<c<b$, $f\, | \, p_2^{-1}\{ c\}$ is homotopic to
	   the restriction of the identity mapping on 
	   $AC_n(Z)$ to $p_2^{-1}\{ c\}$ in $p_2^{-1}((0,1))$. 
\end{itemize}
For $s\in (0,1)$ and $t\in \mathbb{I}$, we define a property
$P(s,t)$ of $H$ as follows: 
\begin{quote}
$H(M_s\cap AC_n(Z), t)\subseteq p_2^{-1}((0,1))$ 
and the restriction of $H(-,t)$ to $M_s\cap AC_n(Z)$
is homotopic to the identity mapping on $M_s\cap AC_n(Z)$ in 
$p_2^{-1}((0,1))$.
\end{quote}
We remark that by the flatness of $H$, if 
$H(M_s\cap AC_n(Z), t)\subseteq p_2^{-1}((0,1))$,  
there is a neighborhood $U$ of $(s;t)$ such that 
$H(M_{s'}\cap AC_n(Z), t')\subseteq p_2^{-1}((0,1))$ for
any $(s';t')\in U$.

\begin{lmm}\label{lmm:flat}
Let $Z$ be a compact path-connected space and $n\in \mathbb{N}$. 
Let $f\! : AC_n(Z)\to AC(Z)$ 
(or $f\! : AC_n(Z)\to CAC(Z)$) be a continuous map.  
If for each $y\in \mathbb{I}$ the set $p_2(f(M_y\cap AC_n(Z)))$ does 
not contain both points
$0$ and $1$, $f$ is homotopic to a flat map $f_1$. 

Moreover, let $H\! : AC_n(Z)\times\mathbb{I}\to AC(Z)$ be a homotopy between
 the identity mapping on $AC_n(Z)$ and $f$ such that for each $y\in \mathbb{I}$ 
and $t\in \mathbb{I}$, the set 
$p_2(H(M_y\cap AC_n(Z),t))$ does not contain both points
$0$ and $1$. 
Then there exists a flat homotopy between the identity mapping and $f_1$. 
\end{lmm}

\begin{proof}
Fix the number $y$. Let $A(y)$ and $B(y)$ be the
minimum and the maximum of the function $p_2\circ f :M_y\cap AC_n(Z)\to
\mathbb{I}$, respectively, and also let 
\[
 C(y) =\frac{A(y)}{1 + A(y) - B(y)}.
\]
We remark that the continuity of $A(y)$ and $B(y)$ 
follows from the fact that the shape of $AC_n(Z)$ is like a direct product. 

Let $\varphi (s_0, s_1,t) = s_0 + (s_1-s_0)t$. 
For $x\in p_2^{-1}((0,1))$ and $u\in \mathbb{I}$, let $x_u\in
 p_2^{-1}(\{ u\})$ be the point determined by the retraction of
 $p_2^{-1}([p_2(x),u])$ or $p_2^{-1}([u,p_2(x)])$ to $p_2^{-1}(\{
 u\})$, which is derived from the direct product structure, 
 and for $x\in p_2^{-1}(\{ 0\}\cup \{ 1\})$ and $u\in \mathbb{I}$,
let $x_u = x$. 
Define $F\! : AC_n \times \mathbb{I} \to AC _n$ by  
\[
 F(x,t) = x_{\varphi (p_2(f(x)), C(p_2(f(x))), t)}. 
\]
Since $A(p_2(f(x)),t)=0$ implies $C(p_2(f(x)),t)=0$ and $B(p_2(f(x)),t)=1$ 
implies $C(p_2(f(x)))=1$ and both cannot occur at the same time, 
$F$ is a deformation retraction, 
$F(-,0)$ is $f$ and $F(-,1)$ is a flat map. 

If $f$ is a flat map, then $f_1 = f$. Since the identity mapping
 on $AC_n$ is a flat map, applying this reasoning to the homotopy $H$ 
we have a homotopy $\ov{H}$ for the second statement, i.e. we 
define $A(y,t), B(y,t), C(y,t)$ using $H(x,t)$ instead of $f(x)$ and let  
\[
 \ov{H}(x,t) = H(x,t)_{\varphi (p_2(H(x,t)), C(p_2(H(x,t),t)), t)}.
\] 
\end{proof}
\begin{lmm}\label{lmm:crucial}
Let $Z$ be a noncontractible space and $H:AC_{2m-1,2m}(Z)\times\mathbb{I}
\to AC(Z)$ a flat homotopy. If $H(M_0\cap AC_{2m-1,2m}(Z),t_0)
\subseteq p_2^{-1}((0,1))$, then
 there exists a neighborhood $U$ of $(0;t_0)$ such that $H$ does not
 satisfy $P(s,t)$ for
 any $(s;t)\in U$ with $s>0$. An analogous statement holds for 
$H(M_1\cap AC_{2m-1,2m}(Z),t_0) \subseteq p_2^{-1}((0,1))$. 
\end{lmm}
\begin{proof}
We have a neighborhood $U$ of $(0;t_0)$ such that 
$H(M_s\cap AC_{2m-1,2m}, t)\subseteq p_2^{-1}((0,1))$ for
any $(s;t)\in U$. We fix $(s;t)\in U$ with $s>0$. 
Let $P_n = (1/n;s)$ and  $I_n = [P_{n+1},P_n]$. 
Then we have $M_s = \{ (0; s)\} \cup \bigcup _{n=1}^\infty 
I_n\cup Z_{P_n}$ and $M_s\cap AC_{2m-1,2m}(Z) = Z_{P_{2m}}\cup Z_{P_{2m-1}}
\cup I_{2m-1}$. 
Since $H(-,t)$ maps $\bigcup _{u\in [P_{2m-1},A_{2m-1}]}Z_u$ into $p^{-1}(\mathbb{I}\times (0,1))$, the restriction of $H(-,t)$ to 
$Z_{P_{2m-1}}$ is null-homotopic in
 $p^{-1}(\mathbb{I}\times (0,1))$. 
Since $M_s$ is a strong deformation retract of 
$p^{-1}(\mathbb{I}\times (0,1))$ and $Z_{P_{2m-1}}$ is a retract of 
$p^{-1}(\mathbb{I}\times (0,1))$ and $Z_{P_{2m-1}}$ is not contractible, 
 the identity mapping on $Z_{P_{2m-1}}$ is not
 homotopic to the restriction of $H(-,t)$ to $Z_{P_{2m-1}}$, which
 implies that $P(s,t)$ does not hold.   

To prove the statement for $H(M_1,t_0)$ we use $Z_{P_{2m}}$ 
and argue on a neighborhood of $B_{2m}$ to obtain a similar conclusion.
\end{proof}

\begin{lmm}\label{lmm:critical}
Let $H:AC_n(\mathbb{H},o)\times\mathbb{I}\to AC(\mathbb{H},o)$ be a flat
 homotopy between the identity mapping on $AC_n(\mathbb{H},o)$ and $f$. 
Then there exist $a_0,b_0 \in \mathbb{I}$ with $a_0<b_0$ such that
 $[a_0,b_0]$ is an essential interval for $f$.
\end{lmm}
\begin{proof}
Let $d: [0,1]\to S^1$ be a winding with the base point $z_0$, i.e. both
 $d | [0,1)$ and $d | (0,1]$ are bijective continuous mappings with
 $d(0)=d(1)=z_0$. 

We define a homotopy $H^{\ast}:S^1\times{\mathbb{I}}\to S^1$ as follows:
 \[
 H^{\ast}(u, t) = \left\{ 
\begin{array}{rl}
d(p_2(H(M_{d^{-1}(u)},t))) &\quad \mbox{ if }u\neq z_0 \mbox{ and }P(d^{-1}(u),t) \mbox{ holds;}\\
z_0, &\quad \mbox{ otherwise. }
\end{array}
\right.
\] 
First we show the continuity of $H^*$.

If $u\neq z_0$ and $P(d^{-1}(u),t)$ holds, the continuity at $(u,t)$ is
 clear. Otherwise, $u\neq z_0$ and $P(d^{-1}(u),t)$ does not hold, or 
$u = z_0$. We consider two cases: 

{\it Case} 1.  Suppose that $u\neq z_0$ and $P(d^{-1}(u),t)$ does not hold: If 
$p_2\circ H(M_{d^{-1}(u)},t) = \{ 0\}$ or $\{ 1\}$, then the
 continuity at $(u,t)$ follows from that of $H$. Otherwise, since
 $H(-,t)$ maps $M_{d^{-1}(u)}\cap AC_n(Z)$ continuously with respect to
 $u$ and $t$, the restriction of $H(-,t)$ to $M_{d^{-1}(u)}\cap AC_n(Z)$
 is not homotopic to the identity mapping on $M_{d^{-1}(u)}\cap AC_n(Z)$ in
 $p^{-1}(\mathbb{I}\times (0,1))$, i.e. $H^*$ takes the value $z_0$ in a
 neighborhood of $(u,t)$. 

{\it Case} 2. Suppose that $u= z_0$: If each of $p_2(H(M_0\cap AC_n(Z),t))$ and
 $p_2(H(M_1\cap AC_n(Z),t))$ 
is equal to either $\{ 0\}$ or $\{ 1\}$, the continuity at $(u,t)$ follows
 from that of $H$. The remaining case is when 
$p_2(H(M_0\cap AC_n(Z),t))\subseteq (0,1)$ or 
$p_2(H(M_1\cap AC_n(Z),t))\subseteq (0,1)$. In this case the
 continuity follows by Lemma~\ref{lmm:crucial}. 

We have shown that $H^{\ast}$ is a homotopy and hence we have 
$a_0,b_0 \in \mathbb{I}$ with $a_0<b_0$ such that
$H^{\ast}(a_0,1) = H^{\ast}(b_0,1) = z_0$, $H^{\ast}(s,1)\neq z_0$ 
for $a_0<s<b_0$ and the orientation of the winding of $H^{\ast}(-,1)$ 
on $[a_0,b_0]$ is the same as that of $d$.   
The last statement implies that $f(a_0) = H(a_0,1) = 0$ and 
$f(b_0) = H(b_0,1) = 1$. 
We have $P(s,1)$ for $a_0<s<b_0$, which implies 
that $[a_0,b_0]$ is an essential interval for $f$.
\end{proof}
{\it Proof of\/ Theorem~\ref{thm:main3} (2)}.
To show the nonequivalence by contradiction suppose that we have 
$f:AC(\mathbb{H},o)\to CAC(\mathbb{H},o)$ and   
$g:CAC(\mathbb{H},o)\to AC(\mathbb{H},o)$ such that 
$g\circ f$ is homotopic to the identity mapping on $AC(\mathbb{H},o)$ 
via a homotopy $H$. 

There exists a sufficiently large $n$ such that the restriction of $H$ to
$AC_n(\mathbb{H},o)\times \mathbb{I}$ satisfies the condition for $H$ in
Lemma~\ref{lmm:flat}. 
Apply Lemma~\ref{lmm:flat} to the homotopy $H$ between the identity 
mapping on
$AC_n(\mathbb{H},o)$ and $g\circ f \, | \, AC_n(\mathbb{H},o)$, to
get a flat homotopy $\ov{H}$ between the identity mapping 
on $AC_n(\mathbb{H},o)$
and a map which is given by applying Lemma~\ref{lmm:flat} 
and homotopic to $g\circ f \, | \, AC_n(\mathbb{H},o)$. 
According to Lemma~\ref{lmm:critical}, we have an essential interval
$[a_0,b_0]$ for $\ov{H}(-,1)$. We remark that $g\circ f
(p_2^{-1}((a_0,b_0))\cap AC_n(\mathbb{H},o)) \subseteq
p_2^{-1}((0,1))$. 

First we claim that $f((1/m;s)) \in L$ for $a_0\le s \le b_0$ and $m\ge
n$. To show this by contradiction, suppose that $f((1/m;s_0))\notin L$. Then
$f((1/m;s_0))$ is contained in a local disk. This implies that the restriction of
$\ov{H}(-,1)$ to $p_2^{-1}(\{ s_0\})$ is not homotopic to the identity
on $p_2^{-1}(\{ s_0\})$, since $\ov{H}(-,1)$ is obtained from $g\circ f$
as in Lemma~\ref{lmm:flat}. Hence $f(\{ 1/m\}\times [a_0,b_0]) \subseteq
L$ for $m\ge n$. Thus, for sufficiently large distinct $m_0, m_1$, $f(\{
1/m_0\}\times [a_0,b_0])$ and $f(\{ 1/m_1\}\times [a_0,b_0])$ are included 
as intervals on $L$. 
Therefore we conclude that some subinterval of $\{
1/m_0\}\times [a_0,b_0]$ 
or $\{ 1/m_1\}\times [a_0,b_0]$ 
is mapped by $g\circ f$ outside of $\{1/m_0\}\times (0,1)$ or 
outside of $\{1/m_1\}\times (0,1)$, i.e. 
at least one of these intervals goes through the base square $\mathbb{I}^2$ 
by the homotopy $\ov{H}$.  
Hence the restriction of $\ov{H}(-,1)$
to $M_s\cap AC_n(Z)$ is not homotopic to the identity on
$p_2^{-1}((0,1))$ for $a_0<s<b_0$, which is a contradiction. 
\qed

\subsection{Proof of Theorem 1.3 (3)}
Next we show that $CSC(\mathbb{H},o)$ and $CAC(\mathbb{H},o)$ are
 homotopy equivalent. 
Actually we show that $CSC(\mathbb{H},o)$ is homotopy equivalent to 
$CSC(S^1,z_0)$ and $CAC(\mathbb{H},o)$ is homeomorphic to
$CAC(S^1,z_0)$. Then the conclusion follows from Theorems~\ref{thm:main1}
and \ref{thm:main2}.
To describe $CSC(\mathbb{H},o)$ and $CAC(\mathbb{H},o)$, we introduce 
the constructions 
$\widetilde{\mathbb{I}\bigvee}_{n\in\mathbb{N}}CSC(X_n,x_n)$ 
and $\widetilde{\mathbb{I}\bigvee}_{n\in\mathbb{N}}CAC(X_n,x_n)$. 

For spaces $X_n$ with $x_n\in X_n$, let 
$Y_n = CSC(X_n,x_n)$ and $Z_n = CAC(X_n,x_n)$. 
First we describe the construction of 
$\widetilde{\mathbb{I}\bigvee}_{n\in\mathbb{N}}CSC(X_n,x_n)$. 
We identify all copies of $(0;s)$ in $Y_n$ for each $s\in
 \mathbb{I}$ and have the quotient set of the disjoint union of $Y_n$'s. 
Then we induce a topology so that any neighborhood of $\{ (0;s) \, | \,
 s\in\mathbb{I}\}$ contains almost all $Y_n$'s. More precisely, let
${}_np$ and ${}_np_2$  be copies of the projections $p$ and $p_2$ for $Y_n$,
 respectively. Then, a basic neighborhood base of $(0;y)$ is 
\[
 \bigcup _{n=1}^k {}_np^{-1}([0,\vep )\times (a,b)) \cup
 \bigcup _{n=k}^\infty  {}_np_2^{-1}((a,b)) 
\] 
for $\vep >0$ and $a<y<b$. 
We denote this space by
$\widetilde{\mathbb{I}\bigvee}_{n\in\mathbb{N}}CSC(X_n,x_n)$. 

We define $\widetilde{\mathbb{I}\bigvee}_{n\in\mathbb{N}}CAC(X_n,x_n)$ 
analogously, i.e.  
we identify all copies of $(0;s)$ in $Z_n$ for each $s\in
 \mathbb{I}$ and have the quotient set of the disjoint union of $Z_n$'s. 
Then we induce a topology so that any neighborhood of $\{ (0;s) \, | \,
 s\in\mathbb{I}\}$ contains almost all $Z_n$'s. 

It is then easy to see that
$CSC((\widetilde{\bigvee}_{n\in\mathbb{N}}(X_n,x_n),x^*)$ is
homeomorphic to 

\noindent
$\widetilde{\mathbb{I}\bigvee}_{n\in\mathbb{N}}CSC(X_n,x_n)$ and 
 $CAC((\widetilde{\bigvee}_{n\in\mathbb{N}}(X_n,x_n),x^*)$ is 
homeomorphic to

\noindent
$\widetilde{\mathbb{I}\bigvee}_{n\in\mathbb{N}}CAC(X_n,x_n)$, where
$x^*$ denotes the point that results from the identification of all $x_n$. 

\smallskip
{\it Proof of Theorem~\ref{thm:main3} (3)}.
We recall the proof of Theorem~\ref{thm:main2}. There, the maps
except $\tau$ assuring homotopy equivalences are level-preserving. 
Moreover $\tau$ can be taken to be $\vep$-level-preserving for 
each $\vep >0$. 
Let $X_n$ be copies of the circle $S^1$.  
Accordingly, we have continuous maps $f_n: CSC(X_n)\to CAC(X_n)$, 
 $g_n: CAC(X_n)\to CSC(X_n)$ and homotopies 
$F_n: CSC(X_n)\times \mathbb{I}\to CSC(X_n)$, 
$G_n: CAC(X_n)\times \mathbb{I}\to CAC(X_n)$ such that: 
\begin{itemize}
\item[(1)] $f_n$ is level-preserving; 
\item[(2)] $g_n$, $F_n$, $G_n$ are $1/n$-level-preserving; and
\item[(3)] $F_n((0;y),t) = (0;y)$ and $G_n((0;y),t)=(0;y)$ for $y,t\in\mathbb{I}$. 
\end{itemize}
Define 
 $f: \widetilde{\mathbb{I}\bigvee}_{n\in\mathbb{N}}CSC(X_n)\to 
\widetilde{\mathbb{I}\bigvee}_{n\in\mathbb{N}}CAC(X_n)$, 
\; $g: \widetilde{\mathbb{I}\bigvee}_{n\in\mathbb{N}}CAC(X_n)\to$ 

\noindent
$\widetilde{\mathbb{I}\bigvee}_{n\in\mathbb{N}}CSC(X_n)$ and homotopies 
$F: \widetilde{\mathbb{I}\bigvee}_{n\in\mathbb{N}}CSC(X_n)\times \mathbb{I}\to \widetilde{\mathbb{I}\bigvee}_{n\in\mathbb{N}} CSC(X_n)$, 
$G: \widetilde{\mathbb{I}\bigvee}_{n\in\mathbb{N}}CAC(X_n)\times \mathbb{I}\to \widetilde{\mathbb{I}\bigvee}_{n\in\mathbb{N}}CAC(X_n)$ as unions of 
$f_n,g_n,F_n,G_n$ respectively. It suffices to verify their continuity. 
We show this for $F$, since the proofs for the others are similar. 
Continuity at points other than $(0;y)$ is obvious. 
For $ \bigcup _{m=1}^{k_0} {}_mp^{-1}([0,\vep )\times (a,b)) \cup
 \bigcup _{m=k_0}^\infty  {}_mp_2^{-1}((a,b)) $ 
with $\vep >0$ and $a<y_0<b$, choose $k\ge k_0$ so that 
$a+1/k< y_0 < b-1/k$. Then, we have 
$a< p_2\circ F_m ((0;y),t)< b$ for $m\ge k$ and $t\in\mathbb{I}$. 
Hence using the continuity of $F_m$ for $m<k$ we get the desired
 neighborhood of 
$(0;y_0)$ which assures the continuity of $F$ at $(0;y_0)$.   
Since $(\mathbb{H},o)$ is
$(\widetilde{\bigvee}_{n\in\mathbb{N}}(X_n,x_n),x^* )$, we conclude that 
$CSC(\mathbb{H},o)$ and $CAC(\mathbb{H},o)$ are homotopy equivalent. 
\qed
\section{Proof of Theorem~\ref{thm:main4}}
We shall consider $2$-cycles in $SC(T,x)$, where $T$
is a $2$-dimensional torus with a distinguished point $x$. 
For this purpose we prove
some basic results about oriented closed surfaces and wedges of tori. 
Let $(T_i,x_i)$ be copies of the torus $(T,x)$ and $r_i: \bigvee
_{i=1}^n (T_i,x_i) \to T_i$ be the retractions, where 
the attaching point of the wedge $\bigvee_{i=1}^n (T_i,x_i)$ is 
denoted by $o$ and $r_i$ maps $T_j$ to $o$ for $j\neq i$.  
We denote the genus of an oriented closed surface $S$ by $g(S)$. 
Let $z$ be a singular $2$-cycle of a space $X$. We can write $z$ as a formal
sum $\Si _{i=0}^m \mu _ig_i$ where $\mu _i = \pm 1$ and $g_i$ are
continuous maps of the 2-simplex $\De _2$ to $X$. 
For a 2-cycle $z$, $[z]$ denotes the homology class containing $z$. 

Since $z$ is a 2-cycle, 
by patching boundaries of copies of $\De _2$ we get an oriented closed
surface $S_z$ and a continuous map $f_z: S_z \to X$. 
Let $[S_z]$ be the homology class of the fundamental cycle of $S_z$. 
Then we have $f_{z*}([S_z]) = [z]$. We refer the reader 
to \cite[pp.108-109]{Hatcher:algtop} for this standard construction. 
 
In case $S_z$ is not connected, we have 2-cycles $z_1,\cdots ,z_k$ such that 
$\Si _{i=1}^kz_i=z$, each $S_{z_i}$ is connected, and $S_z$ is the disjoint 
union of $S_{z_i}$'s. An oriented closed surface $S_z$ constructed 
from a 2-cycle $z$ is generally not unique, but the following results hold 
for any construction of $S_z$. 
\begin{lmm}\label{lmm:basic}
Let $z$ be a singular $2$-cycle of $\bigvee _{i=1}^n (T_i,x_i)$. 
Then the cardinality of $\{ i\; | r_{i*}([z]) \neq 0 \}$ is at most $g(S_z)$. 
\end{lmm}
\begin{proof}
We prove this by induction on the genus $g(S_z)$. By the preceding remark we may suppose that $S_z$ is connected. 
When $g(S_z) = 0$, $S_z$ is a 2-sphere and so $r_i\circ f_z$ is
 null-homotopic for each $i$ hence the conclusion is obvious. 
To prove the induction step by contradiction, suppose to the contrary. 
Without loss of
 generality we may assume that $g(S_z)\ge 1$ and  
$r_{i*}(z) \neq 0$ for $1\le i\le g(S_z)+1$. 
We may also assume that $f_z$ is a piecewise linear map and 
 $f_z^{-1}(\{ o\})$ is a surface with boundary,
 by thickening if necessary.  

We claim that we have a simple closed curve in $f_z^{-1}(\{ o\})$ which is
 essential on $S_z$.  
To show this by contradiction, suppose that $f_z^{-1}(\{ o\})$ does not
 contain any closed curve which is essential on $S_z$. Let $C$ be a connected
 component of
 $f_z^{-1}(\{ o\})$. We have at most finitely many disjoint connected
 boundaries of $C$ which are simple closed curves.  
Since these are inessential, at least one side of this simple closed
 curve on $S_z$ is a disk. If this disk contains $C$ for at least
 two simple closed curves in the boundary of $C$, $S_z$ is a 2-sphere, 
which contradicts $g(S_z)\ge 1$. On the other hand, if 
 this side does not contain $C$ for every simple closed
 curve in the boundary of $C$, $f_z$ is null-homotopic 
in this case, because $\pi
 _2(T_i)$ is trivial for each $i$ and consequently 
$\pi _2(\bigvee _{i=1}^nT_i)$ is trivial. Hence, for exactly one 
 simple closed curve in the boundary of $C$, 
the disk side of this simple closed curve contains $C$. 
Let $D_C$ be this disk in $S_z$ and define $f_C:S_z\to \bigvee
_{i=1}^n (T_i,x_i)$ by  
$f_C(x)=o$ for $x\in D_C$ and $f_C(x) = f_z(x)$ for $x\notin D_C$. 
Then $f_C$ is homotopic to $f_z$ relative to $S_z\sm D_C$. 
The simple closed curves of boundaries of $D_{C_0}$ and $D_{C_1}$ 
 are disjoint 
for distinct components $C_0$ and $C_1$ of $f_z^{-1}(\{ o\})$.
Therefore, if $D_{C_0}$ and $D_{C_1}$ intersect, then 
$D_{C_0}\subseteq D_{C_1}$, $D_{C_1}\subseteq D_{C_0}$ or 
$D_{C_0}\cup D_{C_1} = S_z$. The last case implies that $S_z$ is a 2-sphere, 
which contradicts $g(S_z)\ge 1$. 
Hence we have connected components $C_j (1\le j\le m)$ of
 $f^{-1}(\{ o\})$ such that $D_{C_j}\cap D_{C_{j'}}=\emptyset$ for
 $j\neq j'$ and for each connected component $E$ of $f_z^{-1}(\{ o\})$
 there exists $C_j$ such that $D_E\subseteq D_{C_j}$. 
Now define $\ov{f}: S_z\to \bigvee _{i=1}^n
 (T_i,x_i)$ by $f(x) = o$ for $x\in \bigcup _{j=1}^m D_{C_j}$ and
 $\ov{f}(x)= f_z(x)$ otherwise. 
Then $\ov{f}$ is homotopic to $f_z$. 
Since $D_{C_j}$ are pairwise disjoint disks, 
$S_z \sm \bigcup _{j=1}^m D_{C_j}$ is connected, which implies
that the range of $\ov{f}$ is contained in some $T_i$. 
This contradicts the fact that 
$(r_i\circ f_z)_*$  is essential for $1\le i\le k+1$ with $k\ge 1$. 
 We have shown that there exists an essential closed curve $L$ in
 $f_z^{-1}(\{ o\})$ (we may suppose that this curve is piecewise linear).

We cut open $S_z$ along $L$ and paste two disks. We have a cycle $z'$ 
such that $z'$ is homologous to $z$ and $f_{z'}$ extends $f_z$ so that 
$f_{z'}$ takes the value $o$ on these disks. 
 
{\it Case} 1: $S_{z'}$ is connected.   

Since $g(S_{z'}) = g(S_z)-1$, the cardinality of $\{ i \, :\,
 r_{i*}(z') \neq 0 \}$ 
is at most $g(S_z)-1$. Since $z$ is homologous to $z'$, 
we have a contradiction.

{\it Case} 2: $S_{z'}$ has two connected components. 

We have two cycles $z_0$ and $z_1$ such that $z_0 + z_1$ is homologous to 
$z'$, $S_{z_0}$ and $S_{z_1}$ are connected closed surfaces. Then we have 
$g(S_{z_0}) + g(S_{z_1}) = g(S_z)$. Since $L$ is essential, $g(S_{z_0}), g(S_{z_1})\ge 1$ and consequently 
$g(S_{z_0}),g(S_{z_1})<g(S_z)$. Now there exists $i_0$ such that the both $r_{i_0*}([z_0]) = 0$ and $r_{i_0*}([z_1]) = 0$, which contradicts 
 $r_{i_0*}([z_0]) + r_{i_0*}([z_1]) = r_{i_0*}([z]) \neq 0$.  
 \end{proof}
\begin{cor}\label{cor:basic}
Let $z$ be a singular $2$-cycle of $\widetilde{\bigvee} _{i\in \mathbb{N}} (T_i,x_i)$. 
Then the cardinality of $\{ i\; | r_{i*}([z]) \neq 0 \}$ is at most $g(S_z)$. 
\end{cor}
\begin{proof}
To show this by contradiction, suppose that 
the cardinality of $\{ i\; | r_{i*}([z]) \neq 0 \}$ is greater than $g(S_z)$. 
We have $F\subseteq \mathbb{N}$ such that 
the cardinality of $F$ is greater than $g(S_z)$ and 
 $r_{i*}([z]) \neq 0$ for every $i\in F$.
Let $r_F: \widetilde{\bigvee} _{i\in \mathbb{N}} (T_i,x_i)\to 
\bigvee _{i\in F} (T_i,x_i)$ 
be the retraction projecting every torus $T_j$ with $j \notin F$ to $o$. 
We remark that we can construct $S_{r_{F\#}(z)}$ for $r_{F\#}(z)$ 
with $S_{r_{F\#}(z)} = S_z$. 
Since $r_{i*}([r_{F\#}(z)]) = r_{i*}([z])$ for $i\in F$, 
the cardinality $\{ i\in F: r_{i*}([r_{F\#}(z)]) \neq 0\}$ is greater 
than $g(S_{r_{F\#}(z)})$ which contradicts Lemma~\ref{lmm:basic}.
\end{proof}
{\it Proof of\/} Theorem~\ref{thm:main4}.
First we recall the following Mayer-Vietoris sequence for $SC(T)$:  
\[
 H_2(p_2^{-1}((0,1)))
\stackrel{i_{0*}-i_{1*}}
\longrightarrow 
  H_2(p_2^{-1}([0,1)))\oplus H_2(p_2^{-1}((0,1])) 
\stackrel{j_{0*} + j_{1*}}
\longrightarrow  H_2(SC(T)), 
\]
where $i_0:p_2^{-1}((0,1)) \to p_2^{-1}([0,1))$, 
$i_1:p_2^{-1}((0,1)) \to p_2^{-1}((0,1])$, 
$j_0: p_2^{-1}([0,1))\to SC(T)$ and $j_1: p_2^{-1}((0,1])\to
SC(T)$ are the inclusions. 
We have deformation retractions 
$r_0:p_2^{-1}([0,1))\to M_0$, $r_1:p_2^{-1}((0,1])\to M_1$ and 
$r_{1/2}:p_2^{-1}((0,1))\to M_{1/2}$.  We observe that 
$M_0, M_1$ and $M_{1/2}$ are homotopy equivalent to the Hawaiian 
tori wedge, more precisely, by the local contractiblity of tori we 
have the homotopy equivalences 
\begin{eqnarray*}
M_0 &\simeq& \widetilde{\bigvee}_{n\in\mathbb{N}}p^{-1}(\{ A_n\}) \\
M_1 &\simeq& \widetilde{\bigvee}_{n\in\mathbb{N}}p^{-1}(\{ B_n\}) \\
M_{1/2} &\simeq& \widetilde{\bigvee}_{n\in\mathbb{N}}p^{-1}(\{ C_n\}) 
\end{eqnarray*}
Let $c_{A_k}:M_0\to p^{-1}(\{ A_k\})$ be the retraction so that 
$c_{A_k}(M_0\sm p^{-1}(\{ A_k\})) = \{ A_k\}$ and let 
$c_{B_k}:M_1\to p^{-1}(\{ B_k\})$ and $c_{C_k}:M_{1/2}\to p^{-1}(\{ C_k\})$
be similar retractions (see Figure 1). 
Finally, let $r_{A_k} = c_{A_k}\circ r_0, 
r_{B_k} =c_{B_k}\circ r_1$ and $r_{C_k} =c_{C_k}\circ r_{1/2}$. 

Since $p^{-1}(\{ C_1\})$ is homeomorphic to $T$, we 
let $z$ be the standard cycle such that $[z]$ is the generator of
$H_2(p^{-1}(\{ C_1\})$. 
We'll show that $[z]$ is nontrivial as an element of $H_2(SC(T))$. 
To show this by contradiction we suppose that $[z] = 0$ in $H_2(SC(T))$. 
Since $z$ is a cycle in $p_2^{-1}([0,1))$, we have $[z]_0\in
H_2(p_2^{-1}([0,1))$, where $[ * ]_0$ denotes a homology class in 
$H_2(p_2^{-1}([0,1))$. 
Since $j_{0*}([z]_0) + j_{1*}(0) = [z] =0$, we have 
$u\in H_2(p_2^{-1}((0,1)))$ such that  
$i_{0*}(u) = [z]_0$ and $-i_{1*}(u) = 0$.  

Then we have $r_{A_1*}\circ i_{0*}(u) =  r_{A_1*}([z]_0) = 1$ and 
$r_{A_k*}\circ i_{0*}(u) =  r_{A_k*}([z]_0) = 0$ for $k\ge 2$ and 
$r_{B_k*}\circ i_{1*}(u) = -r_{B_k*}(0) = 0$ for $k\in \mathbb{N}$. 
Consider the above homotopy equivalences of $M_0$ and $M_{1/2}$ together 
with the deformation retractions $r_0$ and $r_{1/2}$. 
Then we can see that $H_2(p^{-1}(\{ C_{2k-2}\}))\oplus 
H_2(p^{-1}(\{ C_{2k-1}\}))$ is a summand of $H_2(p^{-1}((0,1)))$ and 
$H_2(p^{-1}(\{ A_k\}))$ is a summand of $H_2(p^{-1}([0,1)))$. 
Therefore $i_{0*}|H_2(p^{-1}(\{ C_{2k-2}\}))$ and 
$i_{0*}|H_2(p^{-1}(\{ C_{2k-1}\}))$ are isomorphisms 
onto $H_2(p^{-1}(\{ A_k\}))$ for $k\ge 2$ and 
similarly $i_{1*}|H_2(p^{-1}(\{ C_{2k-1}\}))$ and 
$i_{1*}|H_2(p^{-1}(\{ C_{2k}\}))$ are isomorphisms 
onto $H_2(p^{-1}(\{ B_k\}))$ for $k\in \mathbb{N}$ and hence  
we have $r_{C_1*}(u) = 1$ and $r_{C_2*}(u) = -1$ and successively 
$r_{C_{2k+1}*}(u) = 1$ and 
$r_{C_{2k+2}*}(u) = -1$ for $k\in \mathbb{N}$. 
Since $u\in H_2(p_2^{-1}((0,1)))$, we have a singular $2$-cycle $z_0$ of 
$p_2^{-1}((0,1))$ with $[z_0] = u$, where $p_2^{-1}((0,1))$ is homotopy 
equivalent to
$\widetilde{\bigvee}_{i\in \mathbb{N}}(T_i,x_i)$. According to the
description preceding Lemma~\ref{lmm:basic} we have a closed surface
$S_{z_0}$. 
By Corollary~\ref{cor:basic} 
the cardinality of $\{ i \, | \, r_{C_i*}((u))\neq 0\}$ is 
at most $g(S_{z_0})$, which is a contradiction. 

Next we consider the following property of a space:
\begin{quote}
{\it There exists a point $x$ such that any neighborhood of $x$ contains 
an image of a nontrivial $2$-cycle.}
\end{quote}
This property is homotopy invariant. 

Since $p_2^{-1}([0,1))$ and $p_2^{-1}((0,1])$ are contractible by 
Lemma~\ref{lmm:contractible}, 
$CAC(T)$ does not have this property. But for $CSC(T)$ any neighborhood of
a point $(0;y)$ contains an image of a $2$-cycle which is homologous to the
standard $2$-cycle $z$, which is nontrivial. 
We have thus shown that $CSC(T)$ and $CAC(T)$ are not homotopy
equivalent. 
\qed

\begin{rmk}
{\rm For the proof of Theorem~\ref{thm:main4} a weaker assertion than 
Lemma~\ref{lmm:basic} is sufficient.
What is actually necessary is a bound on the cardinality of
$\{ i\; | r_{i*}([z]) \neq 0 \}$ and this bound can be obtained by a
surgery, i.e. by cutting open some $f^{-1}(T_i)$ instead of $f^{-1}(\{ o\})$ 
in the proof of Lemma~\ref{lmm:basic}, which is easier. 
Since Lemma~\ref{lmm:basic} itself is a basic fact, we have stated this
exact form.} 
\end{rmk}
\section{Proof of Theorem~\ref{thm:main5}}

{\it Proof of\/} Theorem~\ref{thm:main5}.
The statement (1) of Theorem~\ref{thm:main5} is a corollary of
\cite[Theorem A.1]{E:free} and so we explain the notions and results
around it. The free $\si$-product of groups, which is a subgroup of the
inverse limit of finite free products of groups \cite{Higman:unrestrict},
is defined by using countably infinite words as a generalization of the
usual free products \cite[Definition 1.2]{E:free}. We refer the reader
to \cite[Definition 1.2]{E:free} for a precise definition. Then for
locally strongly contractible spaces $X_n$, 
$\pi _1(\widetilde{\bigvee}_{n\in\mathbb{N}}(X_n,x_n))$ is isomorphic
to the free $\si$-product of $\pi _1(X_n,x_n)$ by \cite[Theorem
A.1]{E:free}. Thus we have shown (1). 

Next we show (2). Let $o$ be the 
distinguished point of the Hawaiian tori wedge $\mathbb{H}_T$ and $T_n$
be the $n$-th factor of $\mathbb{H}_T$.  
Let $f:\mathbb{I}^2\to \mathbb{H}_T$ be a continuous map
such that $f(\partial \mathbb{I}^2) = \{ o\}$. 
We have the countable family $\mathcal{U}$ consisting of pairwise
disjoint connected open sets such that
$\mathbb{I}^2 \sm f^{-1}(\{ o\}) = \bigcup \mathcal{U}$. 
For each $U\in \mathcal{U}$, $f(U)$ is contained in some $T_n$.  
Since each torus $T_n$ is locally strongly contractible at $o$ and a
neighborhood of $o$ contains almost all $T_n$, we may assume that
$\{ U\in \mathcal{U}: f(U)\subseteq T_n\}$ is finite for each $T_n$ and
also that each $U\in \mathcal{U}$ is an open polygon. 
We remark that $U$ may have
holes and it divides $\mathbb{I}^2$ into finitely many, possibly only one,
connected components. For each $U\in \mathcal{U}$, let $D_U$ be the open
disk such that $\mathbb{I}^2\sm D_U$ is the unique connected component of 
 $\mathbb{I}^2\sm U$ containing $\partial \mathbb{I}^2$. Then we have
 $U\subseteq D_U$. 
For $U,V\in \mathcal{U}$, we define $U\prec V$, if  
$U$ is contained in $D_V$ but $V$ is not contained in $D_U$. 

Since $\prec$ is a partial order and $\mathcal{U}$ is at most countable, 
we have an order-preserving map $\rho : \mathcal{U}\to \mathbb{Q}$. 
The complement of the Cantor ternary set consists of a disjoint union of 
open intervals and the disjoint intervals are ordered naturally 
and this ordering is isomorphic to that of the rationals $\mathbb{Q}$. 
We number the open intervals as $(a_q,b_q)$ by $q\in \mathbb{Q}$ so that 
$p<q$ implies $b_p<a_q$. 
For $U\in \mathcal{U}$ we define $f_U: \ov{D_U}\to \mathbb{H}_T$ such that 
$f_U(x) = f(x)$ for $x\in U$ and $f_U(x) = o$ for $x\in \ov{D_U}\sm U$. 

We define a homotopy $H: \mathbb{I}^2\times \mathbb{I}$ so that $H(x,0) = f(x)$
for $x\in \mathbb{I}^2$, 
$H(x,t) = o$ for $x\in \partial \mathbb{I}^2, t\in \mathbb{I}$ and
$H(x,1) = o$ for $x\in \mathbb{I}^2$. 
For this purpose we define $H$ first on some parts.  

For each $U\in \mathcal{U}$, we define 
$H(x,t) = f(x)$ for $x\in U$ and $t\le a_{\rho (U)}$ 
and let $H\, |\, \ov{D_U}\times [a_{\rho (U)},b_{\rho (U)}]$ to be a homotopy from
$f_U$ to the constant map to $\{ o\}$. 
Now, $H$ is continuous on $\ov{D_U}\times [a_{\rho (U)},b_{\rho (U)}]\cup 
\ov{U}\times [0,a_{\rho (U)}]$ for each $U\in \mathcal{U}$. 

According to the definitions of $\rho$ and $D_U$ 
the defined parts are overlapping only where the value of $H$ is $o$. 
To see this, suppose that $D(U)\cap D(V)\neq \emptyset$ 
for distinct $U$ and $V$. Then we have $D(U)\cap V\neq \emptyset$ or 
$D(V)\cap U\neq \emptyset$.
Since $U$ and $V$ are disjoint and each of them is connected, we 
have $V\subseteq D(U)$ or $U\subseteq D(V)$ and hence $V\prec U$ 
or $U\prec V$. We only deal with the case $V\prec U$. 
Since $U\cap V = \emptyset$ and $b_{\rho (V)} < a_{\rho (U)}$, 
we have 
\[
 (D_U\times [a_{\rho (U)},b_{\rho (U)}]\cup 
U\times [0,a_{\rho (U)}])\cap
 (D_V\times [a_{\rho (V)},b_{\rho (V)}]\cup 
V\times [0,a_{\rho (V)}]) = \emptyset
\]
and we conclude that the defined parts are overlapping only 
where the value of $H$ is $o$. 
 
We define $H(x,t)=o$ for the remaining undefined parts. 
We need to show the continuity of $H$. A crucial matter is 
the accumulation of $H^{-1}(T_n)$ for infinitely many $n$, but we define the 
value of $H$ to be $o$ for such a point. Since each neighborhood of $o$ 
contains almost all $T_n$ and $\{ U\in \mathcal{U}: f(U)\subseteq T_n\}$ is 
finite for each $T_n$, the continuity of $H$ is now clear. 
Since $H(x,1)$ is defined as the 
remaining undefined parts, $H$ is a homotopy to the constant map to $\{ o\}$, 
which implies that $f$ is null-homotopic.  

In the remaining part we shall prove (3). 
Let $f$ be a continuous map from a closed surface $S_z$ to $\mathbb{H}_T$.  
Let $r_i$ be the retraction of $\mathbb{H}_T$ to the $i$-th factor. 
By Corollary~\ref{cor:basic} $r_i\circ f$ is null-homologous for almost all
$i$. 
As in the proof of (2) we may suppose that each connected component of 
$f^{-1}(\mathbb{H}_T\sm \{ o\})$ is an open subsurface. 
Next we find a simple closed curve in $f^{-1}(\{ o\})$ which 
is essential in the surface. 
We cut open along the simple closed curve. We now iterate this procedure. 
  
Hence we have $g_0,\cdots ,g_n$ such that 
\begin{itemize}
\item[(a)] the domain $S_{z_i}$ of each $g_i$ is a connected closed surface; 
\item[(b)] the singular cycle $z_1+\cdots +z_n$ is homologous to $z$; 
\item[(c)] every simple closed curve in $g^{-1}_i(\{ o\})$ 
is null-homotopic in the surface; and  
\item[(d)] each connected component of $g^{-1}_i(\mathbb{H}_T\sm \{ o\})$ 
is an open subsurface. 
\end{itemize}
We fix $g_i$. Let $U$ be a connected component of 
$g^{-1}_i(\mathbb{H}_T\sm \{ o\})$. 
Then by the property (c) every simple closed curve in the boundary of 
$U$ bounds a disk. Since the genus of $S_z$ is positive, at most one of 
those closed disks contains $U$. If there is no such disk, the complement 
of $U$ of the surface is covered by those finitely many disks. Then by Theorem
~\ref{thm:main5}(2) the restriction of $g_i$ to each disk is homotopic to 
the constant $o$ and hence we conclude that $g_i$ is homotopic to 
a map into a single $T_j$ for $j\in \mathbb{N}$. 
Otherwise, for each $U\in\mathcal{U}$ there exists an open disk $D_U$ 
which contains $U$ and whose boundary is a connected component 
of the boundary of $U$. 
Now we define $U\prec V$ if $U\subseteq D_V$, but $V\not\subseteq D_U$   
in a similar way as in the proof of (2). 

If $D_U$ contains $V$ and $D_V$ contains $U$, 
then we have $U=V$. To see this by contradiction, suppose that $U\neq V$. 
Then, since $U\cap V =\emptyset$ and $\partial D_V$ is a connected component 
of $\partial V$, we have $S_{z_i}\sm D_U\subseteq D_V$ and this implies that 
$S_{z_i}$ is a sphere, which is a contradiction. 
Hence $U\prec V$ is a partial order as in the proof of (2). 
In addition, since $S_{z_i}$ is a surface, for distinct $U$ and $V$, 
$D(U)\cap D(V) \neq \emptyset$ implies $D(U)\cap V \neq \emptyset$ or 
$D(V)\cap U \neq \emptyset$ as in (2) and consequently 
$V\prec U$ or $U\prec V$. Now we continue as in the 
proof of (2). The difference is that the domain is a closed 
surface $S_{z_i}$ instead of a square, but the proof is formally the same 
and we have a homotopy from $g_i$ to the constant map to $\{ o\}$. 
Now all of the above implies that $[z] \in 
\oplus _{i\in \mathbb{N}}H_2(T_i) 
\cong \oplus _\mathbb{N}\mathbb{Z}$. 
Since $\oplus _{i\in \mathbb{N}}H_2(T_i)\le H_2(\mathbb{H}_T)$, 
we have 
\[
 H_2(\mathbb{H}_T) = \oplus _{i\in \mathbb{N}}H_2(T_i)\cong  
\oplus _\mathbb{N}\mathbb{Z}. \qquad \qed
\] 
We complete this paper with the following three interesting problems which 
remain open:
\begin{prb} Is the group $H_3(\mathbb{H}_T)$ trivial? 
\end{prb}

\begin{prb} Does there exist a 2-dimensional Peano
continuum which is cell-like, simply connected, noncontractible, and
aspherical in dimension 2?
\end{prb}
\begin{prb} Does there exist a finite-dimensional noncontractible Peano
continuum whose homotopy groups are trivial?
\end{prb}

\section{Acknowledgements}
This research was supported by the Slovenian Research Agency grants
P1-0292-0101-04, J1-4144-0101-07, J1-2057-0101 
and the Polish-Slovenian grant BI-Pl-2010-2011-001. 
The first author was partially supported by the
Grant-in-Aid for Scientific research (C) of Japan
No. 20540097 and 23540110.

\providecommand{\bysame}{\leavevmode\hbox to3em{\hrulefill}\thinspace}

\end{document}